\documentclass[reqno]{amsart}
\usepackage{bm,amsmath,amsthm,amssymb,mathtools,verbatim,amsfonts,tikz-cd,mathrsfs}
\usepackage{enumitem}
\usepackage{float}

\usepackage[hidelinks,colorlinks=true,linkcolor=blue, citecolor=black,linktocpage=true]{hyperref}
\usepackage{graphicx}
\usepackage[alphabetic]{amsrefs}
\usepackage{caption}
\usepackage{morefloats}
\usepackage[top=1.3in, bottom=1.1in, left=1.3in, right=1.3in,marginpar=1in]{geometry}
\usepackage{subfigure}

\usetikzlibrary{matrix,arrows,decorations.pathmorphing}

\usepackage{chngcntr}
\newtheorem{thm}{Theorem}[section]
\newtheorem{prop}[thm]{Proposition}

\newtheorem{cor}[thm]{Corollary}

\theoremstyle{definition}

\theoremstyle{remark}
\newtheorem{rem}[thm]{Remark}

\renewcommand{\d}{\partial}

\renewcommand{\bar}{\overline}

\DeclareMathOperator{\Spin}{{Spin}}

\usepackage{leftidx}

\numberwithin{equation}{section}

\allowdisplaybreaks

\def\dl {\bunderline{d}}
\def\Vu {\overline{V}}
\def\Vl {\sunderline{V}}
\newcommand{\bunderline}[1]{\underline{#1\mkern-2mu}\mkern2mu }
\newcommand{\sunderline}[1]{\underline{#1\mkern-3mu}\mkern3mu }
\def\du {\bar{d}}

\title{A note on PL-disks and rationally slice knots}
\author[K. Hendricks]{Kristen Hendricks}
\thanks{KH was partially supported by NSF grant DMS-2019396 and a Sloan Research Fellowship.}
\address{Department of Mathematics, Rutgers University, New Brunswick, NJ, USA}
\email{kristen.hendricks@rutgers.edu}
\author[J. Hom]{Jennifer Hom}
\thanks{JH was partially supported by NSF grants  DMS-1552285 and DMS-2104144.}
\address{School of Mathematics, Georgia Institute of Technology, Atlanta, GA, USA}
\email{hom@math.gatech.edu}
\author[M. Stoffregen]{Matthew Stoffregen}
\address{Department of Mathematics, Michigan State University, East Lansing, MI, USA}
\email{stoffre1@msu.edu}
\thanks{MS was partially supported by NSF grant DMS-1702532.}
\author[I. Zemke]{Ian Zemke}
\address{Department of Mathematics\\Princeton University\\  Princeton, NJ, USA}
\email{izemke@math.princeton.edu}
\thanks{IZ was partially supported by NSF grant DMS-1703685.}

\begin{document}

\begin{abstract} 
We give infinitely many examples of manifold-knot pairs $(Y, J)$ such that $Y$ bounds an integer homology ball, $J$ does not bound a non-locally-flat PL-disk in any integer homology ball, but $J$ does bound a smoothly embedded disk in a rational homology ball. The proof relies on formal properties of involutive Heegaard Floer homology.
\end{abstract}

\maketitle

\section{Introduction}

Every knot $K$ in $S^3$ bounds a non-locally-flat PL-embedded disk in $B^4$, obtained by taking the cone over $K$. (Throughout, we will not require PL-disks to be locally-flat.) The analogous statement does not hold for knots in more general manifolds. Adam Levine \cite[Theorem 1.2]{Levinenonsurj} found examples of manifold-knot pairs $(Y, J)$ such that $Y$ bounds a contractible 4-manifold and $J$ does not bound a PL-disk in any homology ball $X$ with $\d X = Y$; see also \cite{HLL}.  

The main result of this note concerns rationally slice knots in homology spheres bounding integer homology balls:
%The main result of this note gives examples of manifold-knot pairs $(Y, J)$ such that $Y$ bounds an integer homology 4-ball  and $J$ does not bound a PL-disk in any integer homology ball but $J$ does bound a smoothly embedded disk in a rational homology ball.

\begin{thm}\label{thm:main}
There exist infinitely many manifold-knot pairs $(Y, J)$ where $Y$ is an integer homology sphere and 
\begin{enumerate}
	\item $Y$ bounds an integer homology 4-ball,
	\item $J$ does not bound a PL-disk in any integer homology 4-ball,
	\item $J$ does bound a smoothly embedded disk in a rational homology $4$-ball.
\end{enumerate}
\end{thm}

Throughout, let $Y$ be an integer homology sphere. Recall that a knot $J \subset Y$ is \emph{rationally slice} if $J$ bounds a smoothly embedded disk in a rational homology $4$-ball $W$ with $\d W = Y$. Two manifold-knot pairs $(Y_0, J_0)$ and $(Y_1, J_1)$ are \emph{integrally} (respectively \emph{rationally}) \emph{homology concordant} if $J_0$ and $J_1$ are concordant in an integral (respectively rational) homology cobordism between $Y_0$ and $Y_1$. A knot $J \subset Y$ is integrally (respectively rationally) homology concordant to a knot $K$ in $S^3$ if and only if $J \subset Y$ bounds a PL-disk in an integer (respectively rational) homology ball.

Theorem \ref{thm:main} is an immediate consequence of the following theorem, where $\Vl_0$ and $\Vu_0$ are the involutive knot Floer homology invariants of \cite{HMInvolutive} and $V_0$ the knot Floer homology invariant defined in \cite[Section 2.2]{NiWu} (see also \cite{RasmussenKnots}, \cite{OSKnots}): 

\begin{thm}\label{thm:core}
Let $K$ be a negative amphichiral rationally slice knot in $S^3$ with $\Vl_0 \geq 1$ and $V_0 = \Vu_0 = 0$ and let $\mu$ be the core of surgery in $M=S^3_{-1/\ell}(K)$, where $\ell$ is an odd positive integer. Consider $J = \mu \# U \subset M \# -M$, where $U$ denotes the unknot in $-M$. Then $(M \# -M, J)$ is rationally slice, hence rationally homology concordant to a knot in $S^3$, but $(M \# -M, J)$ is not integrally homology concordant to any knot in $S^3$.
\end{thm}

\begin{rem}
The figure-eight satisfies the hypotheses of Theorem \ref{thm:core} by \cite{Fintushel-Stern:1984-1} (see also \cite[Section 3]{AkbulutLarson} and \cite[Theorem 1.7]{HMInvolutive}). More generally, the genus one knots $K_n$ with $n$ positive full twists in one band and $n$ negative full twists in the other band, $n$ odd, also satisfy the hypotheses of Theorem \ref{thm:core}; see Figure \ref{fig:genusone}. Indeed, $K_n$ is strongly negative amphichiral, hence rationally slice \cite[Section 2]{kawauchi}. Furthermore, $\sigma(K_n)=0$ since $K_n$ is amphichiral. The knot $K_n$ has Seifert form 
\[ \begin{pmatrix}
n & 1 \\
0 & -n
\end{pmatrix} \]
which implies that  $\operatorname{Arf}(K_n) =1$ if and only if $n$ is odd. Since $K_n$ is alternating, it now follows from \cite[Theorem 1.7]{HMInvolutive} that for $n$ odd,  $\Vl_0 =1$ and $V_0 = \Vu_0 = 0$. %∆K(t)=n^2 t^2 −(2n^2 +1)t+n^2
\end{rem}

\begin{figure}[htb!]
\includegraphics{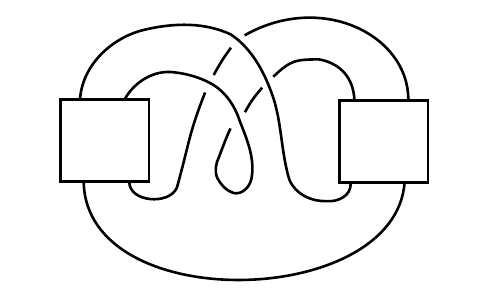}
\put(-115,42){$n$}
\put(-38,42){$-n$}
\caption{\label{fig:genusone} The knot $K_n$, where $n$ and $-n$ denote the number of positive full twists.}
\end{figure}

\begin{rem}
Note that $M$ does not bound an integer homology ball (since, for instance, $\dl(M) =2\Vl_0\neq 0$), but $M \# -M$ does.
\end{rem}

The proof of Theorem \ref{thm:core} is inspired by the proof of \cite[Theorem 1.1(1)]{HLL}. Our proof relies on the following result from \cite{HHSZ} relating the involutive correction term $\dl$ \cite[Section 5]{HMInvolutive} with the ordinary Heegaard Floer correction term $d$ \cite[Section 4]{OSIntersectionForms}, for even denominator surgery on knots in $S^3$:
\begin{prop}[{\cite[Proposition 1.7]{HHSZ}}]\label{prop:HHSZ}
Let $K$ be a knot in $S^3$ and let $p,q>0$ be relatively prime integers, with $p$ odd and $q$ even. Then
\[ \dl(S^3_{p/q}(K), [p/2q]) = d(S^3_{p/q}(K), [p/2q]) \]
where $[p/2q]$ denotes the unique self-conjugate $\Spin^c$ structure on $S^3_{p/q}(K)$.
\end{prop}

\noindent The key feature from the above proposition is that for even denominator surgery on a knot in $S^3$, we have that $\dl$ is equal to $d$ for the unique self-conjugate $\Spin^c$ structure on the surgery. More generally, we have the following corollary of Proposition \ref{prop:HHSZ}:

\begin{cor}\label{cor:HHSZ}
Let $J$ be a knot in an integer homology sphere $Y$ and let $p,q>0$ be relatively prime integers, with $p$ odd and $q$ even.  If $(Y, J)$ is integrally homology concordant to a knot in $S^3$, then 
\[ \dl(Y_{p/q}(J), [p/2q]) = d(Y_{p/q}(J), [p/2q]) \]
where $[p/2q]$ denotes the unique self-conjugate $\Spin^c$ structure on $Y^3_{p/q}(J)$.
\end{cor}

\begin{proof}
If $(Y, J)$ is integrally homology concordant to a knot $(S^3, K)$, then $Y^3_{p/q}(J)$ and $S^3_{p/q}(K)$ are integrally homology cobordant; the homology cobordism is given by surgering along the concordance annulus from $(Y, J)$ to $(S^3, K)$. Since $\dl$ and $d$ are invariants of integer homology cobordism, the result follows from Proposition \ref{prop:HHSZ}.
\end{proof}

The proof of Theorem \ref{thm:core} relies on finding manifold-knot pairs $(Y, J)$ where $\dl$ and $d$ of even denominator surgery along $J$ differ; the result then follows from Corollary \ref{cor:HHSZ}.

\section*{Acknowledgements}
We thank JungHwan Park for helpful conversations and Chuck Livingston for thoughtful comments on an earlier draft.

\section{Proof of Theorem \ref{thm:core}}

\begin{proof}[Proof of Theorem \ref{thm:core}]
We first show that  $(M \# -M, J)$ is rationally slice. Since $K$ is rationally slice, the core of surgery in $M=S^3_{-1/\ell}(K)$ is rationally homology concordant to the core of surgery in $S^3_{-1/\ell}(U)$, which is the unknot in $S^3$; that is, $(M, \mu)$ is rationally slice. Hence $(M \# -M, J)$ is also rationally slice.

We now show that $(M \# -M, J)$ is not integrally homology concordant to any knot in $S^3$. Since $\mu$ is the core of surgery in $S^3_{-1/\ell}(K)$, we have that
\[ M_{1/n}(\mu) = S^3_{1/(n-\ell)}(K). \]
Choose an even positive integer $n$ such that $n >\ell$. Since $\ell$ is odd, $n$ is even, and $n-\ell > 0$, by \cite[Proposition 1.7]{HHSZ} we have that 
\[  \dl(M_{1/n}(\mu) ) =  \dl(S^3_{1/(n-\ell)}(K)) = -2\Vl_0(K) \quad \textup{ and } \quad d(M_{1/n}(\mu)) = d(S^3_{1/(n-\ell)}(K))=-2V_0(K) = 0. \] 
Since $J = \mu \# U \subset M \# -M$, we have that
\[ (M \# -M)_{1/n}(J) = M _{1/n}(\mu) \# -M. \]
Note that $-M = S^3_{1/\ell}(-K)=S^3_{1/\ell}(K)$, where the last equality follows from the fact that $K$ is negative amphichiral. Since $\ell > 0$, \cite[Proposition 1.7]{HHSZ} implies that 
\[ \dl(-M) = -2\Vl_0(K) \quad \textup{ and } \quad d(-M) = \du(-M) = 0.\]
Recall that \cite[Proposition 1.3]{HMZConnectedSum} states that if $Y_1$ and $Y_2$ are integer homology spheres, then
\[ \dl(Y_1 \# Y_2) \leq \dl(Y_1) + \du(Y_2). \]
Hence $\dl(M_{1/n}(\mu) \# -M) \leq -2\Vl_0(K)$. Since $d$ is additive under connected sum, we have that $d(M_{1/n}(\mu) \# -M) = 0$.

We have shown that
\[ \dl((M \# -M)_{1/2}(J)) \leq -2\Vl_0(K) \quad \textup{ and } \quad d((M \# -M)_{1/2}(J)) = 0. \]
Recall that $\Vl_0(K) \geq 1$. Now by Corollary \ref{cor:HHSZ}, it follows that $(M \# -M, J)$ is not integrally homology concordant to any knot in $S^3$.
\end{proof}

\bibliographystyle{custom}
\def\MR#1{}
\bibliography{biblio}

% \bib, bibdiv, biblist are defined by the amsrefs package.
\begin{bibdiv}
\begin{biblist}

\bib{AkbulutLarson}{article}{
      author={Akbulut, Selman},
      author={Larson, Kyle},
       title={Brieskorn spheres bounding rational balls},
        date={2018},
        ISSN={0002-9939},
     journal={Proc. Amer. Math. Soc.},
      volume={146},
      number={4},
       pages={1817\ndash 1824},
         url={https://doi-org.prx.library.gatech.edu/10.1090/proc/13828},
      review={\MR{3754363}},
}

\bib{Fintushel-Stern:1984-1}{incollection}{
      author={Fintushel, Ronald},
      author={Stern, Ronald~J.},
       title={A {$\mu$}-invariant one homology {$3$}-sphere that bounds an
  orientable rational ball},
        date={1984},
   booktitle={Four-manifold theory ({D}urham, {N}.{H}., 1982)},
      series={Contemp. Math.},
      volume={35},
   publisher={Amer. Math. Soc., Providence, RI},
       pages={265\ndash 268},
         url={https://doi.org/10.1090/conm/035/780582},
      review={\MR{780582}},
}

\bib{HHSZ}{misc}{
      author={Hendricks, Kristen},
      author={Hom, Jennifer},
      author={Stoffregen, Matthew},
      author={Zemke, Ian},
       title={Surgery exact triangles in involutive {H}eegaard {F}loer
  homology},
        date={2020},
        note={Preprint, arXiv:2011.00113},
}

\bib{HLL}{misc}{
      author={Hom, Jennifer},
      author={Levine, Adam},
      author={Lidman, Tye},
       title={Knot concordance in homology cobordisms},
        date={2018},
        note={Preprint, arXiv:1801.07770},
}

\bib{HMInvolutive}{article}{
      author={Hendricks, Kristen},
      author={Manolescu, Ciprian},
       title={Involutive {H}eegaard {F}loer homology},
        date={2017},
     journal={Duke Math. J.},
      volume={166},
      number={7},
       pages={1211\ndash 1299},
}

\bib{HMZConnectedSum}{article}{
      author={Hendricks, Kristen},
      author={Manolescu, Ciprian},
      author={Zemke, Ian},
       title={A connected sum formula for involutive {H}eegaard {F}loer
  homology},
        date={2018},
        ISSN={1022-1824},
     journal={Selecta Math. (N.S.)},
      volume={24},
      number={2},
       pages={1183\ndash 1245},
      review={\MR{3782421}},
}

\bib{kawauchi}{article}{
      author={Kawauchi, Akio},
       title={Rational-slice knots via strongly negative-amphicheiral knots},
        date={2009},
        ISSN={1674-5647},
     journal={Commun. Math. Res.},
      volume={25},
      number={2},
       pages={177\ndash 192},
      review={\MR{2554510}},
}

\bib{Levinenonsurj}{article}{
      author={Levine, Adam~Simon},
       title={Nonsurjective satellite operators and piecewise-linear
  concordance},
        date={2016},
     journal={Forum Math. Sigma},
      volume={4},
       pages={Paper No. e34, 47},
         url={https://doi-org.prx.library.gatech.edu/10.1017/fms.2016.31},
      review={\MR{3589337}},
}

\bib{NiWu}{article}{
      author={Ni, Yi},
      author={Wu, Zhongtao},
       title={Cosmetic surgeries on knots in {$S^3$}},
        date={2015},
        ISSN={0075-4102},
     journal={J. Reine Angew. Math.},
      volume={706},
       pages={1\ndash 17},
  url={https://doi-org.proxy.libraries.rutgers.edu/10.1515/crelle-2013-0067},
      review={\MR{3393360}},
}

\bib{OSIntersectionForms}{article}{
      author={Ozsv{\'a}th, Peter~S.},
      author={Szab{\'o}, Zolt{\'a}n},
       title={Absolutely graded {F}loer homologies and intersection forms for
  four-manifolds with boundary},
        date={2003},
     journal={Adv. Math.},
      volume={173},
      number={2},
       pages={179\ndash 261},
}

\bib{OSKnots}{article}{
      author={Ozsv\'ath, Peter},
      author={Szab\'o, Zolt\'an},
       title={Holomorphic disks and knot invariants},
        date={2004},
     journal={Adv. Math.},
      volume={186},
      number={1},
       pages={58\ndash 116},
}

\bib{RasmussenKnots}{thesis}{
      author={Rasmussen, Jacob},
       title={{F}loer homology and knot complements},
        type={Ph.D. Thesis},
        date={2003},
        note={arXiv:math/0306378},
}

\end{biblist}
\end{bibdiv}

\end{document}